\documentclass[11pt]{article}
\usepackage{graphicx}

\usepackage{style}

\title{A note on matchings and co-matchings in bipartite graphs}
\author{Sida Li\thanks{Trinity College, University of Cambridge, United Kingdom. Email: \href{mailto:sl2190@cam.ac.uk}{sl2190@cam.ac.uk}.}}
\date{}

\begin{document}
\maketitle

\begin{abstract}
    A class of bipartite graphs is said to have the strong Erd\H{o}s-Hajnal property if there exists $\eps > 0$ such that every graph $((A, B), E)$ in the class contains a complete or empty induced subgraph with parts $X \subseteq A$, $Y \subseteq B$ where $|X| \ge \eps|A|$ and $|Y| \ge \eps|B|$. Scott, Seymour and Spirkl \cite{scott2023} proved that it is enough to forbid a forest and the bipartite complement of a forest. In this paper, we provide quantitative bounds on $\eps$ when we restrict to matchings. 
\end{abstract}

\section{Introduction}

Let $G = ((A, B), E)$ denote a bipartite graph with parts $A, B$ and with edge set $E$. The \emph{bicomplement} of a bipartite graph $G$ is the graph $G'$ given by $((A, B), (A \times B) \setminus E)$. We call the bicomplement of a complete bipartite graph an \emph{empty} bipartite graph, and the bicomplement of a matching a \emph{co-matching}. Let $G[X, Y]$ denote the subgraph induced on the vertex subsets $(X, Y)$, with $X \subseteq A$ and $Y \subseteq B$, and let such a pair $(X, Y)$ be \emph{pure} if $G[X, Y]$ is complete or empty. We say $G[X, Y]$ has \emph{size} $k$ if $\min\{|X|, |Y|\} = k$. 

The celebrated Erd\H{o}s-Hajnal conjecture \cite{erdos1989} posits that for every graph $H$, there exists a constant $\delta(H) > 0$ such that the following holds: every graph $G$ of order $n$ that does not contain $H$ (as an induced subgraph) has a clique or independent set of size at least $n^{\delta(H)}$. A class $\mathcal{G}$ of graphs is said to have the \emph{strong Erd\H{o}s-Hajnal property}, introduced by Fox and Pach \cite{fox2008}, if there exists $\eps > 0$ such that every $G \in \mathcal{G}$ has a pure pair $(X, Y)$ where $|X| \ge \eps|A|$ and $|Y| \ge \eps|B|$. Alon, Pach, Pinchasi, Radoičić and Sharir \cite{alon2005} proved that for a hereditary class of graphs, this implies every $G \in \mathcal{G}$ has a clique or independent set of size $n^{\delta(H)}$. 

For the class of $H$-free graphs, the strong Erd\H{o}s-Hajnal property forces both $H$ and $\overline{H}$ to be forests, which prescribes $H$ to have at most four vertices. However, Chudnovsky, Scott, Seymour and Spirkl \cite{chudnovsky2020} proved that for any forest $H$, forbidding both $H$ and $\overline{H}$ becomes sufficient for the strong Erd\H{o}s-Hajnal property. Scott, Seymour and Spirkl \cite{scott2023} went on to prove the full analogous result for bipartite graphs: for a family $\mathcal{F}$ of forbidden induced subgraphs, the class of $\mathcal{F}$-free bipartite graphs satisfies the strong Erd\H{o}s-Hajnal property if and only if $\mathcal{F}$ contains a forest and the bicomplement of a forest. 

In this paper, we consider when the forest is a matching and provide quantitative bounds on $\eps$ depending on the sizes of our forbidden matchings and co-matchings. 

\begin{thm}\label{thm:graph-lb}
    Let $G$ be a bipartite graph that does not contain a matching of size $k \ge 2$ and does not contain a co-matching of size $\ell \ge 2$. Then there exists a pure pair $(X, Y)$ with $X \subseteq A$, $Y \subseteq B$ such that:
    \[|X| \ge \frac{4}{13 {k+\ell-4 \choose k-2} - 5}|A| \text{ and } |Y| \ge \frac{4}{13 {k+\ell-4 \choose k-2} - 5}|B|.\]
\end{thm}

In particular, we extract the following bound for the diagonal case. 

\begin{cor}\label{cor:diagonal-lb}
    Let $G$ be a bipartite graph that does not contain a matching or co-matching of size $t \ge 2$. Then there exists a pure pair $(X, Y)$ with $X \subseteq A$, $Y \subseteq B$ such that:
    \[|X| \ge \left(\frac{64\sqrt{\pi}}{13} + o(1)\right) \frac{\sqrt{t}}{4^t}|A| \text{ and } |Y| \ge \left(\frac{64\sqrt{\pi}}{13} + o(1)\right) \frac{\sqrt{t}}{4^t}|B|.\]
\end{cor}

We also provide a corresponding construction with a different exponential base. 

\begin{thm}\label{thm:graph-ub}
    Let $t \ge 2$. For all sufficiently large $m, n$, there exists a bipartite graph $G$ with $|A| = m$, $|B| = n$ such that:
    \begin{outline}
        \1 $G$ does not contain a matching or co-matching of size $t$,
        \1 $G$ has no pure pairs $(X, Y)$ with $X \subseteq A$, $Y \subseteq B$ such that:
        \[|X| \ge \left(\sqrt{\frac{e}{2}} + o(1)\right)\frac{\sqrt{t}}{2^{t/2}}|A| \text{ and } |Y| \ge \left(\sqrt{\frac{e}{2}} + o(1)\right) \frac{\sqrt{t}}{2^{t/2}}|B|.\]
    \end{outline}
\end{thm}

We also investigate off-diagonal behaviour. On the extreme end, a matching of size 2 is isomorphic to its own bicomplement. Substituting $k = 2$ into Theorem \ref{thm:graph-lb}, we obtain a pure pair $(X, Y)$ such that:
\[|X| \ge \frac12 |A| \text{ and } |Y| \ge \frac12 |B|.\]
Furthermore, we demonstrate that equality holds for arbitrarily large $G$. 

Thus, our first interesting case is $k = 3$, whereby examining the bound in Theorem \ref{thm:graph-lb} yields polynomial behaviour in $\ell$. 

\begin{cor}\label{cor:k=3-lb}
    Let $G$ be a bipartite graph that does not contain a matching of size 3 and does not contain a co-matching of size $\ell \ge 2$. Then there exists a pure pair $(X, Y)$ with $X \subseteq A$, $Y \subseteq B$ such that:
    \[|X| \ge \frac{1}{3\ell + \frac13 \log\ell - 3}|A| \text{ and } |Y| \ge \frac{1}{3\ell + \frac13 \log\ell - 3}|B|.\]
\end{cor}

We also provide an algebraic construction, though it leaves room for improvement. 

\begin{thm}\label{thm:k=3-ub}
    Let $\ell \ge 2$. For all sufficiently large $m, n$, there exists a bipartite graph $G$ with $|A| = m$, $|B| = n$ such that:
    \begin{outline}
        \1 $G$ does not contain a matching of size 3 or a co-matching of size $\ell$,
        \1 $G$ has no pure pairs $(X, Y)$ with $X \subseteq A$, $Y \subseteq B$ such that:
        \[|X| \ge \frac{2^{3/5} + o(1)}{\ell^{1/5}} |A| \text{ and } |Y| \ge \frac{2^{3/5} + o(1)}{\ell^{1/5}} |B|.\]
    \end{outline}
\end{thm}

{\bf Log-rank conjecture:} Our work on matchings was initially inspired by the renowned log-rank conjecture. For a Boolean function $f:X \times Y \to \{0, 1\}$, let the \emph{communication complexity} of $f$, denoted by $\cc(f)$, be the minimum number of bits exchanged by the optimal deterministic protocol designed to compute $f$. The \emph{log-rank conjecture}, proposed by Lovász and Saks \cite{lovasz1988}, states that for any Boolean function $f$, $\cc(f)$ is bounded above by a polynomial in the log of its rank, where $r = \rk(f) = \rk(M)$ for the matrix $M \in \{0, 1\}^{X \times Y}$ given by $M_{x, y} = f(x, y)$. 

A consequence of low communication complexity is that $M$ contains a large monochromatic (all 0's or all 1's) submatrix. Nisan and Wigderson \cite{nw1994} proved a recursive reduction that allowed recent progress on the log-rank conjecture to centre around finding large monochromatic submatrices in low-rank binary matrices. 

The best known lower bound, demonstrated by Göös, Pitassi and Watson \cite{goos2015}, provides functions where $\cc(f) = \tilde{\Omega}((\log r)^2)$. Building on the work of Lovett \cite{lovett2016}, Sudakov and Tomon \cite{sudakovtomon2024} established the best known upper bound of $\cc(f) = O(\sqrt{r})$ by proving the following. If $\rk(M) \le r$, then there exist $X \subseteq [m]$ and $Y \subseteq [n]$ where:
\[|X| \ge \frac{m}{2^{O(\sqrt{r})}} \text{ and } |Y| \ge \frac{n}{2^{O(\sqrt{r})}},\]
and $M[X \times Y]$ is monochromatic. 

Here we consider the problem with a weaker notion of rank. Permutation matrices and their complements are of full rank. Thus, given $\rk(M) \le r$, we know that $M$ does not contain a permutation submatrix of size $r+1$ or its complement, which can be interpreted as the canonical full-rank obstructions. By taking $G$ to be the corresponding bipartite graph with $|A| = m$, $|B| = n$, we get the diagonal strong Erd\H{o}s-Hajnal problem above. \\

{\bf Alternatives to rank:} We can compare this to results for other alternatives to rank. Balla, Hambardzumyan and Tomon \cite{balla2025} considered the \emph{factorisation norm} $\gamma_2$, defined by:
\[\gamma_2(M) := \min_{U, V : M = UV} ||U||_{\text{row}} ||V||_{\text{col}},\]
where $||U||_{\text{row}}$ denotes the maximum $\ell_2$-norm of the rows of $U$, and similarly for $||V||_{\text{col}}$. It satisfies $\rk(M) \ge \gamma_2^2(M)$. They prove that for $\gamma_2(M) \le \gamma$, there exists $X \subseteq [m]$ and $Y \subseteq [n]$ where:
\[|X| \ge \frac{m}{2^{O(\gamma^3)}} \text{ and } |Y| \ge \frac{n}{2^{O(\gamma^3)}},\]
and $M[X \times Y]$ is monochromatic. Furthermore, they show this is close to optimal with a construction for $M \in \{0, 1\}^{n \times n}$ with $\gamma_2(M) \le \gamma$ and no monochromatic submatrix of size at least $\frac{n}{2^{\gamma - 3}}$. Notably, this provides a `stronger' weakening of rank compared to ours, though no direct comparison can be made. 

We also consider replacing rank with VC-dimension. Binary matrices with bounded $\gamma_2$ have bounded VC-dimension and we also prove that a bipartite graph without large induced matchings also has bounded VC-dimension. As a strict weakening of both alternatives, the following result does not come as a surprise.

\begin{thm}\label{thm:vc-dim}
    For any $\alpha \in (0,1]$ and sufficiently large $n$, there exists a bipartite graph $G$ with $|A| = |B| = n$, such that:
    \begin{outline}
        \1 the VC-dimension of $G$ is at most 2,
        \1 $G$ has no pure pairs $(X, Y)$ of size $\alpha n$. \\
    \end{outline}
\end{thm}

{\bf Hypergraphs:} The background for our graph-theoretic rank is the generalised log-rank problem for tensors, notably finding a simultaneous weakening of many popular notions of tensor rank. In particular, the identity tensor has ``full rank" in the regimes of slice, partition and geometric rank. Tao \cite{tao2016} proved that the slice rank of a diagonal tensor is the number of non-zero entries, then Naslund \cite{naslund2020} proved the same for partition rank and Kopparty, Moshkovitz and Zuiddam \cite{kopparty2023} for geometric rank. 

We establish that large monochromatic subtensors only arise in sparse or dense binary tensors, providing a contrast to the matrix case. Let $G = ((A_1, \dots, A_k), E)$ denote a $k$-partite, $k$-uniform hypergraph with parts $A_1, \dots, A_k$ and (hyper-)edges $E$. Our main result in this section is the following, where we generalise bicomplements to $k$-partite complements and pure pairs to \emph{pure boxes}. 

\begin{thm}\label{thm:general-hyper}
    Let $\alpha \in (0, 1]$, $k \ge 3$ be constants. For sufficiently large $n$, there exists a $k$-uniform, $k$-partite hypergraph $G$ with $|A_i| = n$ for all $i = 1, 2, \dots, k$, such that:
    \begin{outline}
        \1 $G$ does not contain a matching or co-matching of size 2,
        \1 $G$ has no pure boxes of size $\alpha n$.
    \end{outline}
\end{thm}

Nonetheless, we have a positive result for sufficiently sparse hypergraphs, analogous to a lemma of Gavinsky and Lovett \cite{gavinskylovett2014} for matrices. The precise formulation of Theorem \ref{thm:general-hyper} also tells us $O(n^2)$ edges is necessary to force a linear empty $k$-partite subgraph. 

\begin{thm}\label{thm:sparse-hyper}
    Let $\alpha \in (0, 1]$, $k \ge 2$ be constants and let $t \ge 2$. Let $G$ be a $k$-uniform $k$-partite hypergraph with $|A_i| = n$ for all $i = 1, 2, \dots, k$, satisfying:
    \[|E| \le \frac{(1-\alpha)^2}{4(k-1)(t-1)}n^2.\]
    
    Suppose further that $G$ does not contain a matching of size $t$. Then $G$ contains an empty $k$-partite subgraph of size $\alpha n$. 
\end{thm}

\section{Quantitative diagonal bounds}

For ease of notation, we define the following. 

For $k, \ell \ge 2$, we define \purp{$p(k, \ell)$} to be the largest $\alpha \in [0,1]$ such that the following holds. For all bipartite $G$ which do not contain a matching of size $k$ or a co-matching of size $\ell$, there exists a pure pair $(X, Y)$ with $X \subseteq A$, $Y \subseteq B$ such that:
\[|X| \ge \alpha |A| \text{ and } |Y| \ge \alpha |B|.\]
In the diagonal case, we also define $\purp{p(t)}$ to be $p(t, t)$. 

Let $k \ge 2$. Given a $k$-uniform, $k$-partite hypergraph $G$, define $\purp{\tau(G)}$ to be the size of the largest induced matching or induced co-matching in $G$. 

We turn to proving Theorem \ref{thm:graph-lb}. 

\begin{proof}[Proof of Theorem \ref{thm:graph-lb}]
    We prove a stronger bound on $|X|, |Y|$. In particular, consider the two-parameter sequence $s_{k, \ell}$ for $k, \ell \ge 2$, defined by the recurrence relation: $s_{2, \ell} = s_{k, 2} = 2$, and for $k, \ell \ge 3$, 
    \[s_{k, \ell} = \frac{(s_{k-1, \ell} + s_{k, \ell-1}) + 2 + \sqrt{(s_{k-1, \ell} + s_{k, \ell-1})^2 + 4}}{2}.\]
    
    It suffices to prove:
    \[p(k, \ell) \ge \frac{1}{s_{k, \ell}}.\]
    
    Indeed, for $w \ge 4$, we have:
    \[w+1 < \frac{w + 2 + \sqrt{w^2 + 4}}{2} \le w + \frac54.\]
    
    Thus, $s_{k, \ell} \le s_{k-1, \ell} + s_{k, \ell-1} + \frac54$. Setting $x_{k, \ell} := s_{k, \ell} + \frac54$ and comparing to binomial coefficients, we obtain:
    \[s_{k, \ell} \le \frac{13}{4} {k+\ell-4 \choose k-2} - \frac{5}{4}.\]
    
    By Stirling's approximation, in the diagonal case this indeed yields:
    \[s_{k, \ell} \le \left(\frac{13}{64} + o(1)\right) \frac{4^t}{\sqrt{t\pi}}.\]

    We proceed by induction on $k + \ell$. Let $A = [m]$, $B = [n]$. \\

    \emph{Base case:} suppose $k = 2$ or $\ell = 2$ (which includes $k + \ell \le 5$). Note that an induced matching of size 2 is isomorphic to a co-matching of size 2, hence either induced condition is equivalent to $G$ not containing an induced copy of $e_1 \sqcup e_2$ where $e_i$ are vertex-disjoint edges. Thus, given any $x, x' \in A$, we cannot have that both $N(x)\setminus N(x')$ and $N(x') \setminus N(x)$ are non-empty. Hence, either $N(x) \subseteq N(x')$ or $N(x') \subseteq N(x)$, and similarly for $y, y' \in B$. 

    Order the vertices in $A$ as $x_1, \dots, x_m$ such that $N(x_i) \subseteq N(x_{i+1})$ for $i = 1, 2, \dots, m-1$. Identically we obtain $y_1, \dots, y_n$. Note that if $x_i \sim y_j$, then for any $i' \ge i$, $j' \ge j$, we have $y_j \in N(x_{i'})$ hence $x_{i'} \in N(y_j) \subseteq N(y_{j'})$ and $x_{i'} \sim y_{j'}$. 

    If $x_{\ceil{m/2}} \sim y_{\ceil{n/2}}$, then $G[\{\ceil{m/2}, \dots, m\}, \{\ceil{n/2}, \dots, n\}]$ is complete. Otherwise, $x_{\ceil{m/2}} \not \sim y_{\ceil{n/2}}$ and $G[\{1, \dots, \ceil{m/2}\}, \{1, \dots, \ceil{n/2}\}]$ is empty. This gives us:
    \[|X| \ge \frac{m}{2} \text{ and } |Y| \ge \frac{n}{2}.\]

    \emph{Inductive step:} let $k + \ell \ge 6$ and suppose for the sake of contradiction that $G$ does not contain a matching of size $k$, a co-matching of size $\ell$ and a pure pair $(X, Y)$ with:
    \[|X| \ge \frac{1}{s_{k, \ell}}|A| \text{ and } |Y| \ge \frac{1}{s_{k, \ell}} |B|.\]

    For ease of notation, define:
    \[\lambda := \frac{s_{k-1, \ell}}{s_{k, \ell}}, \quad \mu := \frac{s_{k, \ell-1}}{s_{k, \ell}}.\]
    
    Pick any $\alpha, \beta \in [\mu, 1 - \lambda]$. Since $s_{k,\ell} > s_{k-1, \ell} + s_{k, \ell-1}$, this interval is non-empty. Partition $A = A_\alpha^+ \sqcup A_\alpha^-$ where:
    \[A_\alpha^+ = \{x \in A : |N(x)| \ge \alpha n\}.\]
    Similarly, $B = B_\beta^+ \sqcup B_\beta^-$ with $B_\beta^+ = \{y \in B : |N(y)| \ge \beta m\}$. 

    Suppose $x \sim y$ with $x \in A_\alpha^-$, $y \in B_\beta^-$, and let $H = G[A\setminus N(y), B \setminus N(x)]$. Note that:
    \[|A \setminus N(y)| > m - \beta m \ge \lambda m,\]
    and similarly $|B \setminus N(x)| > (1-\alpha)n \ge \lambda n$. If $H$ contains a matching of size $k-1$, say $H[S, T]$, then $G[S \cup \{x\}, T \cup \{y\}]$ is an induced matching of size $k$, contradiction. Also, $H$ cannot contain a co-matching of size $\ell$. Thus, by the inductive hypothesis, there exists $X' \subseteq A \setminus N(y) \subseteq A$ and $Y' \subseteq B \setminus N(x) \subseteq B$ such that:
    \[|X'| \ge \frac{1}{s_{k-1, \ell}} \cdot |A \setminus N(y)| \ge \frac{1}{s_{k, \ell}}m,\]
    \[|Y'| \ge \frac{1}{s_{k-1, \ell}} \cdot |B \setminus N(x)| \ge \frac{1}{s_{k, \ell}}n,\]
    and $G[X', Y'] = H[X', Y']$ is complete or empty. This yields a contradiction. \\

    Otherwise, for all $x \in A_\alpha^-$, $N(x) \subseteq B_\beta^+$. We identically repeat for all $y \in B_\beta^-$. Thus, all of $A_\alpha^-$ and $B_\beta^-$ are non-neighbours and $G[A_\alpha^-, B_\beta^-]$ is empty. 

    Similarly, suppose $x \not \sim y$ for $x \in A_\alpha^+$, $y \in B_\beta^+$ and let $H = G[N(y), N(x)]$. $H$ cannot contain a co-matching of size $\ell - 1$ or a matching of size $k$. Since $|N(y)| \ge \beta m \ge \mu m$ and $|N(x)| \ge \alpha n \ge \mu n$, we can conclude as above with the inductive hypothesis. Otherwise, for all $x \in A_\alpha^+$, $B' \setminus N(x) \subseteq B_\beta^-$, i.e. all of $A_\alpha^+$ and $B_\beta^+$ are neighbours and $G[A_\alpha^+, B_\beta^+]$ is complete, contradiction. 

    Thus, we arrive at a contradiction if:
    \[|A_\alpha^+| \ge \frac{1}{s_{k, \ell}} m, |B_\beta^+| \ge \frac{1}{s_{k, \ell}}n \text{ or } |A_\alpha^-| \ge \frac{1}{s_{k, \ell}} m, |B_\beta^-| \ge \frac{1}{s_{k, \ell}}m.\]
    
    Hence, for all values of $\alpha, \beta \in [\mu, 1-\lambda]$, we have both:
    \[\left(|A_\alpha^+| < \frac{1}{s_{k,\ell}}m \text{ or } |B_\beta^+| < \frac{1}{s_{k, \ell}}n\right) \text{ and } \left(|A_\alpha^-| < \frac{1}{s_{k, \ell}}m \text{ or } |B_\beta^-| < \frac{1}{s_{k, \ell}}n \right).\]

    First suppose $\alpha = \beta = 1 - \lambda$, thus without loss of generality, we can suppose $|A_{1  - \lambda}^-| \le \frac{1}{s_{k, \ell}}m$ (here $A, B$ play symmetric roles). Thus,
    \[|A_{1  - \lambda}^+| \ge \left(1 - \frac{1}{s_{k, \ell}}\right)m > \frac{1}{2}m > \frac{1}{s_{k, \ell}}m,\]
    hence for all $\beta \in [\mu, 1-\lambda]$, $|B_\beta^+| \le \frac{1}{s_{k, \ell}}n$ which gives:
    \[|B_\beta^-| > \left(1 - \frac{1}{s_{k, \ell}}\right)n.\]
    
    Take $\beta = \mu$. We double-count edges:
    \begin{align*}
        |A_{1-\lambda}^+| \cdot \underbrace{(1 -\lambda)n}_{\mathclap{\substack{\text{lower bound on}\\\text{degree in $A_{1-\lambda}^+$}}}} + (m - |A_{1-\lambda}^+|) \cdot \underbrace{0}_{\mathclap{\substack{\text{lower bound on}\\\text{degree in $A_{1-\lambda}^-$}}}} \le &|E| \le |B_\mu^-| \cdot \underbrace{\mu m}_{\mathclap{\substack{\text{upper bound on}\\\text{degree on $B_\mu^-$}}}} + (n - |B_\mu^-|) \cdot \underbrace{m}_{\mathclap{\substack{\text{upper bound on}\\\text{degree on $B_\mu^+$}}}} \\
        \Rightarrow |A_{1-\lambda}^+| \cdot (1-\lambda) n \le &|E| \le nm - |B_\mu^-| \cdot (1 - \mu)m.
    \end{align*}
    
    Hence with the above bounds:
    \[\left(1 - \frac{1}{s_{k, \ell}}\right)m \cdot (1 - \lambda)n < nm - \left(1 - \frac{1}{s_{k, \ell}}\right)n \cdot (1-\mu)m,\]
    which gives:
    \[\left(1-\frac{1}{s_{k, \ell}}\right) (2-\lambda - \mu) < 1.\]
    
    Yet $s_{k, \ell}$ is chosen to be the larger solution of:
    \[\left(1 - \frac{1}{s}\right) \left(2 - \frac{s_{k-1,\ell} + s_{k, \ell-1}}{s}\right) = 1,\]
    giving a contradiction. 
\end{proof}

\subsection{Upper bound construction}

For $p \in [0,1]$, let $G(m, n, p)$ denote an Erd\H{o}s-Rényi random bipartite graph: pick edges between parts of size $m, n$ independently, uniformly with probability $p$. We alter such a graph to provide a construction close to Theorem \ref{thm:graph-lb}. 

\begin{proof}[Proof of Theorem \ref{thm:graph-ub}]
    Consider $G' \sim G(L, L, 1/2)$ for $L = (1 - \frac{2\log t}{t}) \sqrt{\frac{2}{e}} \sqrt{t} 2^{t/2}$. For each $X \subseteq A$, $|Y| \subseteq B$ with $|X| = |Y| = t$, define the event:
    \[A_{X, Y} = \{G'[X, Y] \text{ is a matching, co-matching, complete or empty}\}.\]
    
    Let $\mathcal{A}$ denote the collection of all such events. We have:
    \[\mathbb{P}(A_{X, Y}) = 2 \cdot t! \frac{1}{2^{t^2}} + 2\cdot \frac{1}{2^{t^2}} = 2(t! + 1)\frac{1}{2^{t^2}}.\]
    
    Construct a dependency graph $\mathcal{G}$ on $\mathcal{A}$ by having:
    \[A_{X, Y} \sim A_{X', Y'} \iff |X \cap X'| \ge 1 \text{ and } |Y \cap Y'| \ge 1.\]
    
    Indeed, if $|X \cap X'| = 0$, then no pair $(x, y)$ lies in both $X \times Y$ and $Y \times Y'$, hence the edges/non-edges in $G'[X, Y]$ and $G'[X', Y']$ are independent. Note that $\mathcal{G}$ is regular with degree at most:
    \[\Delta \le {t \choose 1}{L \choose t-1} \cdot {t \choose 1}{L \choose t-1} \le t^2 \left(\frac{L^{t-1}}{(t-1)!}\right)^2 = t^4\frac{L^{2t-2}}{(t!)^2}.\]
    
    Since $t! \ge t^t/e^{t-1}$, we have:
    \[e\prob{A_{X, Y}}(\Delta + 1) \le 3et^4 \frac{L^{2t-2}}{t! 2^{t^2}} \le 3t^4 \left(\frac{eL^{2 - 2/t}}{t2^t}\right)^t \to 0.\]
    
    Thus, by the Lovász Local Lemma, there exists $G'$ for which no $A_{X,Y}$ hold, i.e. $\tau(G') \le t-1$ and $G'$ as no pure pair of size $t$. 

    Now let $G$ denote the equitable blow-up of $G'$ to $|A| = m$, $|B| = n$ (namely within $A$ and $B$ separately, vertex blow-ups are as equal in size as possible). No two vertices in an induced matching or co-matching have the same neighbourhood, hence must lie in distinct vertex blow-ups. Thus, $\tau(G) = \tau(G') \le t-1$. 

    Furthermore, the largest pure pair in $G$ arises from the blow-up of a pure pair in $G'$. Indeed, if $x, x'$ lie in the same vertex blow-up in $A$, they share the same neighbourhood hence for any pure pair $(X, Y)$ with $x \in X$, $(X \cup \{x'\}, Y)$ is still a pure pair. In particular, for all $X \subseteq A$, $Y \subseteq B$ with $|X| \ge \ceil{\frac{m}{L}} t$ and $|Y| \ge \ceil{\frac{n}{L}} t$, we have that $(X, Y)$ cannot be a pure pair. 
\end{proof}

We can rephrase Theorem \ref{thm:graph-lb}, Corollary \ref{cor:diagonal-lb} and Theorem \ref{thm:graph-ub} in terms of $p(k, \ell)$. 

\begin{cor}\label{cor:p-diag}
    For $k, \ell \ge 2$, we have:
    \[p(k, \ell) \ge \frac{4}{13 {k+\ell-4 \choose k-2} - 5}.\]
    
    For $t \ge 2$, we have:
    \[\left(\frac{64\sqrt{\pi}}{13} + o(1)\right) \frac{\sqrt{t}}{4^t} \le p(t) \le \left(\sqrt{\frac{e}{2}} + o(1)\right) \frac{\sqrt{t}}{2^{t/2}}.\]
\end{cor}

\section{Off-diagonal bounds}

For the general off-diagonal scenario, we have for fixed $k$ that $s_{k, \ell} = O_k(\ell^{k-2})$.

\begin{cor}\label{cor:fixed-k}
    For fixed $k$, 
    \[p(k, \ell) = \Omega_k(\ell^{-(k-2)}).\]
\end{cor}

Note that in the case of $k = 2$, we demonstrated in the \emph{Base case} of the proof of Theorem \ref{thm:graph-lb} that there exists a pure pair $(X, Y)$ with:
\[|X| \ge \frac12 |A| \text{ and } |Y| \ge \frac12 |B|.\]

A matching construction with $|A| = |B| = 2$ is the half-graph with bipartite adjacency matrix:
\[\begin{pmatrix}
    1 & 1 \\
    0 & 1
\end{pmatrix}.\]

Equitable blow-ups of this, as in the proof of Theorem \ref{thm:graph-ub}, demonstrate equality for arbitrarily large class sizes. 

\begin{cor}\label{cor:k=2}
    For $\ell \ge 2$, we have:
    \[p(2, \ell) = \frac12.\]
\end{cor}

To deal with $k = 3$, we refer back to the recurrence.

\begin{proof}[Proof of Theorem \ref{cor:k=3-lb}]
    Let $a_\ell = s_{3, \ell}$. We have that $a_2 = 2$ and:
    \[\left(1 - \frac{1}{a_\ell}\right) \left(2 - \frac{a_{\ell - 1} + 2}{a_\ell}\right) = 1,\]
    which rearranges to give:
    \[a_\ell - a_{\ell - 1} = 3 + \frac{1}{a_\ell - 1}.\]
    In particular, $a_\ell$ is increasing and $a_2 = 2$, hence $a_\ell - a_{\ell - 1} \ge 3$. This gives $a_\ell \ge 3\ell-4$. 

    Now consider $x_\ell = a_\ell - 3\ell$. Then:
    \[x_\ell -x_{\ell - 1} \le \frac{1}{3\ell - 5}.\]
    
    Thus, we obtain by telescoping:
    \[x_n = x_2 + \sum_{\ell = 3}^n (x_\ell - x_{\ell - 1}) \le -4 + \sum_{\ell = 3}^n \frac{1}{3\ell - 5}.\]
    
    We can place an upper bound on the sum via an integral:
    \[x_n \le -4 + \int_2^n \frac{1}{3x-5} dx = -4 + \frac13 \log(3n-5).\]
    
    In particular, we get $a_\ell \le 3\ell + \frac13\log \ell - 3$. 
\end{proof}

\subsection{Upper bound construction}

We also provide a construction exhibiting polynomial growth but with a different exponent, with its roots in finite geometry. \\

Let a \emph{generalised quadrangle} be an incidence structure $\mathcal{Q} = (P, \mathcal{L}, \text{I})$, where $P$ is a set of points, $\mathcal{L}$ is a set of lines and $I \subseteq P \times \mathcal{L}$ is an incidence relation which satisfies the following conditions:
\begin{outline} 
    \1[(i)] any two distinct points are incident with at most one common line,
    \1[(ii)] for every point $x \in P$ and every line $l \in \mathcal{L}$ not incident with $x$, there is a unique line $m \in \mathcal{L}$ incident with $x$ such that $m$ \emph{meets} $l$ (meaning they are incident with a common point $y \in P$),
    \1[(iii)] every point is incident with at least three lines and conversely every line is incident with at least three points.
\end{outline}
Say $x, y \in P$ are \emph{collinear} if they lie on a common line from $\mathcal{L}$. Then condition (ii) can be rephrased as: there exists a unique point $y$ incident with $l$ such that $x$ is collinear with $y$. 

A generalised quadrangle has \emph{order} $(s, t)$ if every line is incident with exactly $s+1$ points and every point is incident with exactly $t+1$ lines. Its incidence graph is the bipartite graph with vertex classes $P$ and $\mathcal{L}$, where a point is adjacent exactly to the lines incident with it. Note that the conditions force this graph to have girth at least eight. Indeed, a four-cycle would give two distinct lines through the same two points, contradicting (i), and a six-cycle contradicts uniqueness within (ii). 

Let $q$ be a prime power and let $V = \mathbb{F}_q^4$ be equipped with a non-degenerate alternating bilinear form $\braket{\cdot, \cdot}: V \times V \to \mathbb{F}_q$. The \emph{symplectic generalised quadrangle} $W(3, q)$ is the incidence structure, with associated projective space $\text{PG}(3, q)$, constructed as follows. Its points are the one-dimensional subspaces of $V$, its lines are two-dimensional totally isotropic subspaces of $V$ and incidence is given by containment. Thus, a two-dimensional subspace $L \le V$ is a line of $W(3, q)$ precisely when $\braket{x, y} = 0$ for all $x, y \in L$. 

The incidence structure $W(3, q)$ is a generalised quadrangle of order $(q, q)$. We quote standard facts about $W(3, q)$ from Payne and Thas \cite{payne2009}. 

\begin{lmm}\label{lmm:w3q}
    Let $q$ be a prime power. There exists a bipartite graph $H_q = ((P, \mathcal{L}), E)$ such that:
    \begin{outline}
        \1 $|P| = |\mathcal{L}| = N_q = (q+1)(q^2 + 1)$,
        \1 $H_q$ is $(q+1)$-regular and has girth at least eight,
        \1 for all $X \subseteq P$, $Y \subseteq \mathcal{L}$, we have:
        \[\left|e_{H_q}(X, Y) - \frac{q+1}{N_q}|X||Y| \right| \le \sqrt{2q}\sqrt{|X||Y|}. \tag{$\star$}\]
    \end{outline}
\end{lmm}
\begin{proof}
    Let $H_q$ denote the incidence graph of $W(3, q)$, then the first two assertions are standard. 

    Let $M$ be its point-line incidence matrix and let $A$ be the adjacency matrix of the corresponding collinearity graph. Two distinct points lie on at most one common line hence $MM^T = (q+1)I + A$. We have that the collinearity graph is strongly regular with parameters:
    \[(n, r, e, f) = ((q + 1)(q^2 + 1), q(q+1), q-1, q+1).\]
    
    Thus, we obtain eigenvalues $q(q+1)$, $q-1$ and $-1-q$. It follows that the singular values of $M$ are $q+1$, $\sqrt{2q}$ and 0. 

    The third assertion now follows from the biregular variant of the expander mixing lemma. 
\end{proof}

For ease of notation, we define:
\[L_q := \floor{(q^2+1)(\sqrt{2q} + 1)} + 1, \quad \eta_q := \frac{\sqrt{2q}}{q+1}.\]

First we find an infinite family of graphs that saturate Theorem \ref{thm:k=3-ub}.

\begin{thm}\label{thm:w3q-complement}
    Let $q$ be a prime power. There exists a bipartite graph $G_q$ with $|A| = |B| = N_q$, such that:
    \begin{outline}
        \1 $G_q$ does not contain a matching of size 3 or a co-matching of size $L_q$,
        \1 $G_q$ has no pure pairs $(X, Y)$ with $X \subseteq A$, $Y \subseteq B$ such that:
        \[|X| > \eta_q |A| \text{ and } |Y| > \eta_q|B|.\]
    \end{outline}
\end{thm}
\begin{proof}
    Let $H_q$ be given by Lemma \ref{lmm:w3q}, and let $G_q$ be its bicomplement. 

    If $G_q$ contained an induced matching of size 3, then the corresponding induced subgraph of $H_q$ on the same six vertices forms a six-cycle, which contradicts the second property of $H_q$. 

    Suppose that $G_q$ contains a co-matching of size $s$ given by $X \subseteq P$, $Y \subseteq \mathcal{L}$. Then $H_q[X, Y]$ is an induced matching of size $s$. By the third property in Lemma \ref{lmm:w3q}, 
    \[\left|s - \frac{q+1}{N_q}s^2\right| \le \sqrt{2q} \cdot s,\]
    hence:
    \[s \le \frac{N_q}{q+1}(1 + \sqrt{2q}) = (q^2+1)(1 + \sqrt{2q}) < L_q.\]

    Finally, consider a pure pair $(X, Y)$ in $G_q$. Then $H_q[X, Y]$ is complete or empty. 

    \emph{Case 1:} If $H_q[X, Y]$ is complete, note that $H_q$ has no $K_{2, 2}$ subgraph hence $\min\{|X|, |Y|\} \le 1$.

    \emph{Case 2:} If $H_q[X, Y]$ is empty, then using the third property again gives:
    \[\frac{q+1}{N_q} |X||Y| \le \sqrt{2q} \sqrt{|X||Y|}.\]
    Thus, $\sqrt{|X| |Y|} \le \eta_q N_q$, giving $\min\{|X|, |Y|\} \le \eta_qN_q$. 
\end{proof}

Now we interpolate among this family.

\begin{proof}[Proof of Theorem \ref{thm:k=3-ub}]
    Let $x = \left(\frac{\ell}{\sqrt{2}}\right)^{2/5}$. Choose $q$ to be the largest prime such that:
    \[q < x\left(1 - \frac{1}{\log x}\right).\]
    
    By the prime number theorem, we have $q = (1+o(1))x$. Note that:
    \[L_q = \sqrt{2}q^5 (1 + O(q^{-1/2})).\]
    Thus, we have $L_q \le \ell$ for sufficiently large $\ell$. 

    Consider $G_q$ given by Theorem \ref{thm:w3q-complement}. It does not contain a matching of size 3 or a co-matching of size $L_q$, hence certainly not one of size $\ell$. Furthermore, it has no pure pairs $(X, Y)$ with proportion at least:
    \[\eta_q = \sqrt{2} q^{-1/2} (1 + O(q^{-1})) = \sqrt{2} \left((1+o(1)) \left(\frac{\ell}{\sqrt{2}}\right)^{2/5}\right)^{-1/2} = \frac{2^{3/5} + o(1)}{\ell^{1/5}}.\]

    Taking equitable blow-ups gives arbitrarily large graphs. 
\end{proof}

We combine Corollary \ref{cor:k=3-lb} and Theorem \ref{thm:k=3-ub} into bounds on $p(3, \ell)$. 

\begin{cor}\label{cor:k=3}
    For $\ell \ge 2$, we have:
    \[\frac{1}{3\ell + \frac13 \log \ell - 3} \le p(3, \ell) \le \frac{2^{3/5} + o(1)}{\ell^{1/5}}.\]
\end{cor}

\subsection{Other forbidden families}\label{sec:families}

As noted in the introduction, Scott, Seymour and Spirkl \cite{scott2023} proved that forbidding a forest and a bicomplement of a forest is necessary and sufficient for the strong Erd\H{o}s-Hajnal property. Notably, this eliminates just matchings or co-matchings of size at least 3, as well as the half-graph. Taking a random bipartite graph with high girth (or its complement) provides such a construction. 

Now we turn to VC-dimension. We first relate it to induced matchings. 

\begin{lmm}\label{lmm:vc-dim}
    Suppose $G$ is a bipartite graph that does not contain a matching of size $r+1$. Then $G$ has VC-dimension at most $r$. 
\end{lmm}
\begin{proof}
    Suppose for the sake of contradiction it has VC-dimension at least $r+1$, say $X$ is shattered by $\{N(x)\}_{x \in V(G)}$ for $X = \{x_1, \dots, x_{r+1}\}$. In particular, $X = N(x) \cap X$ for some $x$, hence $X$ must lie entirely within one part of $G$, say $A$.

    Now there exists $Y = \{y_1, \dots, y_{r+1}\} \subseteq B$ such that $N(y_i) \cap X = \{x_i\}$, i.e. $G[X, Y]$ is an induced matching of size $r+1$, giving a contradiction.  
\end{proof}

With the construction forbidding a matching of size 3 (since its bicomplement is not a forest), we get Theorem \ref{thm:vc-dim}. 

\section{Generalisation to hypergraphs}

For $k$-uniform, $k$-partite hypergraphs, the situation rapidly degenerates as demonstrated below. 

\begin{thm}\label{thm:hypergraph-sr1}
    Let $\alpha \in (0,1]$, $k \ge 2$ be constants. For sufficiently large $n$, there exists a $k$-uniform $k$-partite hypergraph $G$ with $|A_i| = n$ for all $i$, satisfying the following properties:
    \begin{outline}
        \1 $G$ has at most:
        \[\left(\frac{kH(\alpha)}{\alpha^k} + o(1)\right)n^2\]
        (hyper-)edges, where $H(\alpha) = -\alpha\log_2(\alpha) - (1-\alpha)\log_2(1-\alpha)$,
        \1 $G$ contains no empty $k$-partite subgraph of size $\alpha n$,
        \1 its adjacency tensor $A(G) \in \mathbb{R}^{n \times \dots \times n}$ has slice rank 1.
    \end{outline}
\end{thm}
\begin{proof}
    Consider the Erd\H{o}s-Rényi random $k$-uniform $k$-partite hypergraph on vertex sets $A_1 \sqcup \dots \sqcup A_k$ with $|A_i| = n$ for $i = 1, \dots, k$, and edges chosen with probability $p = Cn^{-k+1}$ where:
    \[C = C_{\alpha, k} > \frac{kH(\alpha)}{\alpha^k}.\]

    Let $M$ denote the number of empty $k$-partite subgraphs of size $\alpha n$. We have as $n \to \infty$, 
    \[\mathbb{E}[M] = {n \choose \alpha n}^k (1-p)^{\alpha^k n^k} \le {n \choose \alpha n}^k e^{-p\alpha^k n^k} = \exp(kH(\alpha)n + o(n) - C\alpha^k n) \to 0,\]
    using entropy bounds for binomial coefficients. 

    By the Chernoff bound, 
    \[\prob{|E| \ge (1+\delta) pn^k} \le \left(\frac{e^{\delta}}{(1+\delta)^{1+\delta}}\right)^{pn^k} \to 0\]
    as well, hence for sufficiently large $n$, we can find $G'$ with $M = 0$ and:
    \[|E(G')| \le (1+o(1)) pn^k = (1+o(1)) Cn.\]

    We perform an alteration. For each edge $\{x_1, \dots, x_k\}$ with $x_i \in A_i$ for $i = 1, \dots, k$ and each $y_k \in A_k$, add in the edge $\{x_1, \dots, x_{k-1}, y_k\}$. Let $G$ denote the resultant hypergraph. The adjacency tensor of $G$ is of the form:
    \[M_{x_1, \dots, x_k} = 1_S(x_1, \dots, x_{k-1}), \quad x_i \in \{0, 1\},\]
    hence has slice rank 1.

    Furthermore, adding new edges doesn't create new empty $k$-partite subgraphs. Thus, $G$ satisfies:
    \[|E(G)| \le |E(G')| \cdot n = (1+o(1))Cn^2\]
    and contains no empty $k$-partite subgraphs of size $\alpha n$. 
\end{proof}

This gives us Theorem \ref{thm:general-hyper}.

\begin{proof}[Proof of Theorem \ref{thm:general-hyper}]
    Take $G$ constructed from Theorem \ref{thm:hypergraph-sr1}. It has $O(n^2)$ edges hence cannot contain a complete subgraph of size $\alpha n$, which has $\alpha^k n^k$ edges, for $k \ge 3$ and sufficiently large $n$. 

    Furthermore, note that the adjacency tensors of the matching and co-matching of size two have slice rank two. Indeed, let $M \in \{0, 1\}^{2 \times \dots \times 2}$ denote the adjacency tensor of the matching, say. Thus, 
    \[M_{x_1, \dots, x_k} = \begin{cases}
        1 & x_1 = \dots = x_k \\
        0 & \text{otherwise},
    \end{cases}\]
    with slice rank clearly at most two. 
    
    Suppose for the sake of contradiction $M$ had slice rank one. Without loss of generality (by symmetry), suppose:
    \[M_{x_1, \dots, x_k} = 1_S(x_1, \dots, x_{k-1}).\]
    
    Then:
    \[1 = M_{0,\dots,0,0} = M_{0,\dots,0,1} = 0,\]
    contradiction. Identically, the adjacency tensor of the co-matching of size two also has slice rank 2. 
    
    Since slice rank is monotone, $G$ cannot contain a matching or co-matching of size two. 
\end{proof}

It turns out $n^2$ is the correct threshold for the number of edges to observe this behaviour. 

\begin{proof}[Proof of Theorem \ref{thm:sparse-hyper}]
    Let $S_1, \dots, S_k$ denote the vertices in each part with degree at most $\frac{1-\alpha}{2(k-1)(t-1)}n$. We have:
    \[|A_i\setminus S_i| \le \frac{|E(G)|}{\frac{(1-\alpha)}{2(k-1)(t-1)}n} \le \frac{1-\alpha}2 n\]
    hence $|S_i| \ge \frac{1+\alpha}{2}n$ for $i = 1, \dots, k$. Suppose for the sake of contradiction that all induced subgraphs of size $\alpha n$ contain an edge. We inductively construct a matching of size $t$ within $G[S_1, \dots, S_k]$. 

    Since $|S_i| \ge \frac{1+\alpha}{2}n \ge \alpha n$ for each $i$, it contains an edge. For the inductive step, assume we have $T_i \subseteq S_i$ with $|T_i| = s \le t-1$ and $G[T_1, \dots, T_k]$ is an induced matching. Let $T_i' \subseteq S_i$ denote all vertices in the $i$th part which are adjacent to some vertex in at least one of $T_1, \dots, T_{i-1}, T_{i+1}, \dots, T_k$. In particular, $T_i \subseteq T_i'$. Then:
    \[|T_i'| \le \underbrace{(k-1) \cdot s}_{\substack{\mathclap{\text{number of vertices in}}\\\text{$T_1, \dots, T_{i-1}, T_{i+1}, \dots, T_k$}}} \frac{1-\alpha}{2(k-1)(t-1)}n \le \frac{1-\alpha}{2}n.\]
    
    Now $|S_i \setminus T_i'| \ge \frac{1+\alpha}{2}n - \frac{1-\alpha}{2}n = \alpha n$, hence $G[S_1\setminus T_1', \dots, S_k \setminus T_k']$ contains an edge, say $\{x_1, \dots, x_k\}$. We can verify that $G[T_1 \cup \{x_1\}, \dots, T_k \cup \{x_k\}]$ is an induced matching of size $s+1$. 
\end{proof}

\subsection{Linear hypergraphs}

\emph{Linear hypergraphs} have edges that intersect in at most one vertex. Under the same assumptions as above, these are almost entirely empty. 

\begin{thm}\label{thm:linear-matchings}
    Let $k \ge 3$ be constant and let $t \ge 2$. Let $G$ be a $k$-uniform $k$-partite linear hypergraph with $|A_i| = n \ge (t-1)^{k-1}$ for all $i = 1, \dots, k$, such that $G$ does not contain a matching of size $t$. Then $G$ has an empty $k$-partite subgraph of size $n-(t-1)^{k-1}$. 
\end{thm}
\begin{proof}
    We follow the structure of the proof of Theorem \ref{thm:sparse-hyper}. Suppose for the sake of contradiction that all induced subgraphs of size $n-(t-1)^{k-1}$ contain an edge. We again inductively construct a matching of size $t$ within $G$. 

    To start, $G$ contains an edge. For the inductive step, assume we have $T_i \subseteq A_i$ with $|T_i| = s \le t-1$ such that $G[T_1, \dots, T_k]$ is an induced matching. Fix index $j$ and let $T_j' \subseteq A_j \setminus T_j$ denote the vertices in $A_j$ such that for each $y_j \in T_j'$, there exists $x_i \in T_i$ for each $i \neq j$ such that $\{x_1, \dots, x_{j-1}, y_j, x_{j+1}, \dots, x_k\}$ forms an edge.

    For each choice of $(x_1, \dots, x_{j-1}, x_{j+1}, \dots, x_k)$, at most one $y_j$ exists since $G$ is linear. Furthermore, there are $s$ disjoint edges $(x_1, \dots, x_k)$ already. Thus, $|T_j'| \le s^{k-1} - s \le (t-1)^{k-1} - s$ hence $|A_j \setminus (T_j \cup T_j')| \ge n - (t-1)^{k-1}$. 

    By assumption, $G[A_1 \setminus (T_1 \cup T_1'), \dots, A_k \setminus (T_k \cup T_k')]$ contains an edge, say $\{y_1, \dots, y_k\}$. We claim that $G[T_1 \cup \{y_1\}, \dots, T_k \cup \{y_k\}]$ is an induced matching of size $s+1$. 
    
    \emph{Case 1:} If such a $k$-tuple contains none of $y_1, \dots, y_k$, then we reduce to $G[T_1, \dots, T_k]$. 
    
    \emph{Case 2:} If it contains one such $y_j$, by construction of $T_j'$, it is not an edge.
    
    \emph{Case 3:} If it contains at least two and is not the edge $\{y_1, \dots, y_k\}$, then it intersects this edge in at least two vertices. Since $G$ is linear, it cannot be an edge either. 
    
    Thus, the only additional edge is $\{y_1, \dots, y_k\}$. 
\end{proof}

One can generalise this proof to arbitrary vertex class sizes.

\begin{thm}\label{thm:linear-unequal}
    Let $t \ge 2$, $k \ge 3$ be constants. Let $G$ be a $k$-uniform, $k$-partite linear hypergraph with $|A_i| = n_i \ge (t-1)^{k-1}$ for all $i = 1, \dots, k$, such that $G$ does not contain a matching of size $t$. Then $G$ has an empty $k$-partite subgraph with parts of size:
    \[n_i - (t-1)^{k-1}, \quad i = 1, \dots, k.\]
\end{thm}

By the monotonicity of slice rank and the slice rank of a diagonal tensor from Tao \cite{tao2016}, we obtain the following corollary.

\begin{cor}
    Let $k \ge 3$ and $r \ge 1$ be constants. Let $G$ be a $k$-uniform $k$-partite linear hypergraph with vertex classes $|A_i| = n_i \ge r^{k-1}$ for $i = 1, \dots, k$, such that the adjacency tensor of $G$ has slice rank $r$. Then $G$ has an independent set of size:
    \[\sum_{i = 1}^k n_i - kr^{k-1}.\]
\end{cor}

\section{Concluding remarks}

As noted in the introduction, the strong Erd\H{o}s-Hajnal property holds when both a forest and its bicomplement are forbidden. Our quantitative bounds in Theorems \ref{thm:graph-lb} and \ref{thm:graph-ub} are catered to induced matchings and co-matchings, hence do not provide an immediate generalisation to other natural families of forests (or in the context of the log-rank conjecture, full-rank acyclic matrices). 

Furthermore, taking our forbidden family to be all full rank $r \times r$ submatrices gives exactly the log-rank conjecture, hence conjecturally $\eps$ behaves like $1/2^{\text{polylog(r)}}$ instead of exponential. It would be interesting to find a proof which works for other patterns and to classify families giving various growth rates of $\eps$. \\

{\bf Open problems:} We provide suggestions for further investigation. 
\begin{outline}[enumerate]
    \1 What is the true base of the exponential in $p(t)$?
    \1 What is the true polynomial growth of $p(3,\ell)$? 
    \1 To concretely generalise to other families: let $\mathcal{H} = \{H_1, H_2, \dots\}$ denote a family of bipartite forests where $H_t$ has order $t$ and $H_t$ is an induced subgraph in $H_{t+1}$. Define $p_\mathcal{H}(t)$ to be the analogue of $p(t)$ for $\mathcal{H}$, where one forbids induced $H_t$ and its bicomplement. For which families is $p_H(t)$ exponential in $t$?
\end{outline}

\section*{Acknowledgments} 

The author is grateful to Cosmin Pohoata for introducing the problem and multiple fantastic discussions. The author is also grateful to Gaia Carenini for many helpful suggestions, including an improvement to Theorem \ref{thm:k=3-ub} (originally a random construction). The author would like to thank Marcelo Campos for communicating the idea behind the construction in Theorem \ref{thm:graph-ub}. 

Part of this research was conducted during the 2025 Baruch College Discrete Mathematics REU. The author is grateful to the Global Talent Fund for making participation in the REU possible and to Jane Street for supporting the REU. The author is also grateful to Adam Sheffer, Guy Moshkovitz and Imre Leader for their continued guidance and mentorship. 

\bibliographystyle{plain}
\bibliography{citations.bib}

\end{document}